\documentclass[12pt]{amsart}
\usepackage{amssymb,bbold}

\theoremstyle{plain}
\newtheorem{theorem}{Theorem}[section]
\newtheorem{proposition}[theorem]{Proposition}
\newtheorem{corollary}[theorem]{Corollary}

\theoremstyle{definition}

\newtheorem{example}[theorem]{Example}

\newcommand{\abs}[1]{\lvert#1\rvert}
\newcommand{\norm}[1]{\lVert#1\rVert}
\newcommand{\bigabs}[1]{\bigl\lvert#1\bigr\rvert}
\newcommand{\bignorm}[1]{\bigl\lVert#1\bigr\rVert}

\newcommand{\Bignorm}[1]{\Bigl\lVert#1\Bigr\rVert}
\newcommand{\term}[1]{{\textit{\textbf{#1}}}}

\renewcommand{\mid}{\::\:}

\def\one{\mathbb 1}
\DeclareMathOperator{\Range}{Range}

\begin{document}

\title{Spaces of regular abstract martingales}

\author{V.G. Troitsky}
\author{F. Xanthos}

\address
   {Department of Mathematical
   and Statistical Sciences, University of Alberta, Edmonton,
   AB, T6G\,2G1. Canada}
\email{troitsky@ualberta.ca,foivos@ualberta.ca}

\thanks{The authors were supported by NSERC grants}

\date{\today}

\maketitle
\begin{abstract}
  In \cite{Troitsky:05,Korostenski:08}, the authors introduced and studied the space $\mathcal M_r$ of regular martingales on a vector lattice and the space $M_r$ of bounded regular martingales on a Banach lattice. In this note, we study these two spaces from the vector lattice point of view. We show, in particular, that these spaces need not be vector lattices. However, if the underlying space is order complete then $\mathcal M_r$ is a vector lattice and $M_r$ is a Banach lattice under the regular norm. 
\end{abstract}

\section{The space of regular martingales on a vector lattice}

Let $F$ be a vector lattice. A sequence $(E_n)$ of positive projections on $F$ such that $E_nE_m=E_{n\wedge m}$ is said to be a \term{filtration}. We will try to impose as few additional assumptions on the filtration as possible. A sequence $X=(x_n)$ in $F$ is a \term{martingale} with respect to the filtration $(E_n)$ if $E_nx_m=x_n$ whenever $m\ge n$.  A sequence $X=(x_n)$ in $F$ is a \term{supermartingale} if $E_nx_m\le x_n$ whenever $m\ge n$ (note that in our definition we do not require that $E_nx_n=x_n$). We denote with $\mathcal{M}=\mathcal{M}\bigl(F,(E_n)\bigr)$ the space of all martingales on $F$ with respect to the filtration $(E_n)$. The space $\mathcal{M}$ equipped with the coordinate-wise order is an ordered vector space and we denote with $\mathcal{M}_+$ the positive cone of $\mathcal{M}$. There is an extensive literature on abstract martingales on vector and Banach lattices, see, e.g., \cite{DeMarr:66,Uhl:71,Kuo:05,Troitsky:05,Korostenski:08,Labuschagne:10,Gessesse:11,Grobler:14,Gao:14,Gao:14a}. For unexplained terminology on ordered vector and Banach spaces we refer the reader to \cite{Aliprantis:06,Aliprantis:07,Meyer-Nieberg:91}.

The space of \term{regular martingales} is defined as follows:
\begin{displaymath}
  \mathcal{M}_r=\mathcal{M}_r\bigl(F,(E_n)\bigr)
   =\bigl\{X_1-X_2 \mid X_1,X_2 \in \mathcal{M}_+\bigr\}.
\end{displaymath}
Equivalently, a martingale $X$ is regular iff $\pm X\le Y$ for some positive martingale $Y$. This definition is motivated by the definition of a regular operator.  In this setting, it is a well known fact that the space of regular operators $\mathcal L_r(F)$ is itself a vector lattice when $F$ is order complete. So it has been a natural conjecture that $\mathcal M_r$ is a vector lattice whenever $F$ is order complete. We will prove that this is indeed the case. This result improves \cite[Theorem 2.3]{Korostenski:08}, which asserts that
$\mathcal{M}_r$ is a vector lattice (and is even order complete) provided that $F$ is order complete and for every $n$ the projection $E_n$ is order continuous and $\Range E_n$ is an order complete sublattice of $F$. 

It has also been an open question whether $\mathcal M_r$ is \emph{always} a vector lattice. We will present an example to the contrary.

\begin{theorem}\label{Mr-VL}
  Let $F$ be an order complete vector lattice and $(E_n)$ a filtration
  on $F$. Then $\mathcal M_r$ is an order complete vector lattice.
\end{theorem}

\begin{proof}
  Let $\mathcal A$ be a subset of $\mathcal{M}_r$ such that $\mathcal A$ is bounded from above in $\mathcal{M}_r$. We will show that $\sup\mathcal A$ exists in $\mathcal M_r$. Let $\mathcal S$ be the set of all supermartingales $Y$ that dominate $\mathcal A$, that is, $X\le Y$ for all $X\in\mathcal A$. By assumption, $\mathcal S$ is non-empty.
  For every $n$, put
  \begin{displaymath}
    z_n=\inf\bigl\{y_n\mid Y=(y_k)\in\mathcal S\bigr\}.
  \end{displaymath}
  We claim that $Z=(z_n)$ is a martingale and $Z=\sup\mathcal A$.

  First, observe that $X\le Z$ for each $X=(x_n)\in\mathcal A$. Indeed, for each $n$ and each $Y=(y_n)\in\mathcal S$ we have $x_n\le y_n$, so that $x_n\le z_n$.

  Next, observe that $Z$ is a supermartingale. Let $m\ge n$, then for
  every $Y\in\mathcal S$ we have $z_m\le y_m$, so that $E_nz_m\le
  E_ny_m\le y_n$. It follows that $E_nz_m\le z_n$.

  Next, we will show that $Z$ is, in fact, a martingale. Fix $k\in\mathbb
  N$ and define $Y=(y_n)$ as follows:
  \begin{displaymath}
     Y=(E_1z_k,E_2z_k,\dots,E_{k-1}z_k,z_k,z_{k+1},z_{k+2},\dots).
  \end{displaymath}
  We claim that $Y$ is a supermartingale. Indeed, let $n\le m$.
  \begin{eqnarray*}
    \text{If}\quad k\le n&\text{ then }&E_ny_m=E_nz_m\le z_n=y_n;\\
    \text{if}\quad m<k&\text{ then }&E_ny_m=E_nE_mz_k=E_nz_k=y_n;\\
    \text{if}\quad n<k\le m&\text{ then }&E_ny_m=E_nz_m=E_nE_kz_m\le E_nz_k=y_n.\\
  \end{eqnarray*}
  Next, note that $X\le Y$ for each $X=(x_n)\in\mathcal A$. Indeed, if $n\ge k$ then
  $y_n=z_n\ge x_n$. If $n<k$ then $x_k\le z_k$ implies $E_nx_k\le E_nz_k$, so that $x_n\le y_n$.

  This yields that $Y\in\mathcal S$, so that $Z\le Y$. It follows that
  for every $n<k$ we have $z_n\le y_n=E_nz_k$, so that $z_n\le
  E_nz_k$. Therefore, $z_n=E_nz_k$ for all $n$ and $k$ with $n<k$. Also note that
  \begin{math}
    E_nz_n=E_nE_nz_{n+1}=E_nz_{n+1}=z_n
  \end{math}
  for every $n$.  Thus, $Z$ is a martingale and $X\le Z$ for each
  $X\in\mathcal A$. Clearly, every martingale $Y$ dominating $\mathcal
  A$ is in $\mathcal S$ and, therefore, $Z\le Y$. Hence,
  $Z=\sup\mathcal A$.

  Let $X\in\mathcal M_r$. Applying the previous argument with $\mathcal A=\{\pm X\}$, we conclude that $\abs{X}$ exists. It follows that $\mathcal M_r$ is a vector lattice. By the preceding computation, it is order complete.
\end{proof}

\begin{example}
  \emph{$\mathcal{M}_r$ need not be a vector lattice}.
  Let $F=c$. For each $n$, define $E_n\colon F\to F$ as follows: if $x=(\alpha_i)$ we put
  \begin{multline*}
    E_nx=
    \bigl(\alpha_1,\dots,\alpha_{3n},
     \tfrac{\alpha_{3n+1}+\alpha_{3n+2}}{2},
     \tfrac{\alpha_{3n+1}+\alpha_{3n+2}}{2},\alpha_{3n+3},\\
     \tfrac{\alpha_{3n+4}+\alpha_{3n+5}}{2},
     \tfrac{\alpha_{3n+4}+\alpha_{3n+5}}{2},\alpha_{3n+6},
    \dots\bigr).
  \end{multline*}
  It is easy to see that $(E_n)$ is a filtration on $c$ and each $E_n$ is order continuous. It is a \term{dense} filtration in the sense that
  \begin{math}
    \bigcup_{n=1}^\infty\Range E_n
  \end{math}
  is dense in $F$.  Define $(x_n)$ as follows:
  \begin{displaymath}
    x_n=(\underbrace{1,-1,0,1,-1,0,\dots,1,-1,0}_{3n},0,0,\dots).
  \end{displaymath}
  Note that $X=(x_n)$ is a martingale with respect to $(E_n)$; it is regular because $\pm x_n\le\one$ for every $n$, where $\one$ is the constant one sequence. We will write $x_{n,i}$ for the $i$-th coordinate of $x_n$. We claim that $X$ has no modulus in $\mathcal{M}_r$. Indeed, suppose that $\pm X\le Y$ for some martingale $Y$, $Y=(y_n)$. For each $n$ we have $y_n\ge\pm x_n$, so that
  \begin{displaymath}
    y_n\ge(\underbrace{1,1,0,1,1,0,\dots,1,1,0}_{3n},0,0,\dots)=u_n.
  \end{displaymath}
  It follows that
  \begin{displaymath}
    y_1=E_1y_n\ge E_1u_n
       =u_n.
  \end{displaymath}
  Since $n$ is arbitrary, it follows that $y_{1,3k+1}\ge 1$ and  $y_{1,3k+2}\ge 1$ for every $k$. Since $y_1$ is an element of $c$, there is $k_0$ such that $y_{1,3k_0}>0$. Define a martingale $Z=(z_n)$ as follows: for every $n$ and $i$, put $z_{n,i}=y_{n,i}$ except when $i=3k_0$, in this case put $z_{n,3k_0}=0$ (for every $n$). It is easy to see that $Z$ is a martingale and $\pm X\le Z<Y$. It follows that $X$ has no modulus.
\end{example}

\section{Krickeberg's formula}

Once again using the analogy with regular operators, we recall that if $F$ is an order complete vector lattice then $\mathcal L_r(F)$ is a vector lattice and the lattice operations on $\mathcal L_r(F)$ are given by the Riesz-Kantorovich formula. There is a similar formula for lattice operations on $\mathcal M_r$. Let $X=(x_n)$ be a martingale. It has been observed in the literature that that the modulus $\abs{X}$ of a regular martingale $X=(x_n)$ often satisfies the following identity:
\begin{displaymath}
  \abs{X}_n=\sup_{m\ge n} E_n\abs{x_m}.
\end{displaymath}
For classical martingales, this identity goes back to Krickeberg's decomposition (see i.e., \cite[p.~32]{Meyer:72}); in the following we will refer to it as \term{Krickeberg's formula}. If the Krickeberg's formula is valid for $F$, that is, if the modulus of every martingale in $\mathcal M_r\bigl(F,(E_n)\bigr)$ is given by the Krickeberg's formula, then, clearly, $\mathcal M_r$ is a vector lattice and the other lattice operations are given by similar formulae; see, e.g., \cite[Theorem~7]{Troitsky:05}.

In the following proposition, we summarize several cases where $\mathcal{M}_r$ is a vector lattice and the lattice operations are given by Krickeberg's formula; it extends \cite[Theorem~2.3]{Korostenski:08}, \cite[Proposition 11]{Troitsky:05}, and \cite[Proposition 4]{Gessesse:11}.

\begin{proposition}\label{Krickeberg}
  Let $F$ be a vector lattice and $(E_n)$ a filtration on $F$. Suppose that any of the following hold.
  \begin{enumerate}
  \item\label{oc-oc} $F$ is order complete and each $E_n$ is order continuous;
  \item\label{fin-rank} $F$ is Archimedean and $E_n$ is of finite rank for each $n$;
  \item\label{lat-hom} $E_n$ is a lattice homomorphism for each $n$.
  \end{enumerate}
  Then $\mathcal{M}_r$ is a vector lattice and the lattice operations are given by the Krickeberg's formula.
\end{proposition}

\begin{proof}
  Let $X=(x_n)$ be a martingale such that $\pm X\le Y$ for some positive martingale $Y=(y_n)$. For a fixed $n$, the sequence $\bigl(E_n\abs{x_m}\bigr)_{m=n}^\infty$ is increasing in $m$, bounded below by $\abs{x_n}$ and above by $y_n$. Indeed, if $n\le m$ then
\begin{displaymath}
  E_n\abs{x_m}=E_n\bigabs{E_mx_{m+1}}\le E_nE_m\abs{x_{m+1}}=E_n\abs{x_{m+1}},
\end{displaymath}
\begin{displaymath}
  \abs{x_n}=\bigabs{E_nx_n}\le E_n\abs{x_n},\quad\text{and}\quad
  E_n\abs{x_m}\le E_ny_m=y_n.
\end{displaymath}

  \eqref{oc-oc} Since $F$ is order complete, $\sup\limits_{m\ge n}E_n\abs{x_m}$ exists for every $n$. Denote it $z_n$ and put $Z=(z_n)$. Clearly, $\pm X\le Z\le Y$. Note that $Z$ is a martingale: for every $k\le n$, since $E_k$ is order continuous, we have
  \begin{displaymath}
    E_kz_n=E_k\bigl(\sup_{m\ge n}E_n\abs{x_m}\bigr)
    =\sup_{m\ge n}E_kE_n\abs{x_m}=\sup_{m\ge n}E_k\abs{x_m}=z_k.
  \end{displaymath}
  It follows that $Z=\abs{X}$. Note that $Z$ is given by the Krickeberg's formula.

  \eqref{fin-rank} Let $H_n=\Range E_n$ and $H_n^+=F^+\cap H_n$. Since $H_n$ is finite-dimensional, we may view it as an ordered Banach space. Note that $\bigl(E_n\abs{x_m}\bigr)_{m=n}^\infty$ and $y_n$ are in $H_n^+$. Since $F$ is Archimedean and $H_n$ is finite-dimensional, $H_n^+$ is closed (in $H_n$) by \cite[Corollary~3.4]{Aliprantis:07} and normal by \cite[Lemma~3.1]{Aliprantis:07}. It follows from \cite[Theorem~2.45]{Aliprantis:07} that
\begin{math}
  \lim_mE_n\abs{x_m}=z_n
\end{math}
in $H_n$. Since $H_n^+$ is closed, it follows from
\begin{math}
  \abs{x_n}\le E_n\abs{x_m}\le y_n
\end{math}
that $\abs{x_n}\le z_n\le y_n$. 

  Repeating this process for every $n$, we produce a sequence $Z=(z_n)$ in $F_+$. To show that $Z$ is a martingale, let $k\le n$. Since $E_k$ is a continuous operator on $H_n$, we have
\begin{displaymath}
  E_kz_n=E_k\bigl(\lim_{m\to\infty}E_n\abs{x_m}\bigr)=
    \lim_{m\to\infty}E_kE_n\abs{x_m}=
    \lim_{m\to\infty}E_k\abs{x_m}=z_k,
\end{displaymath}
where the limit is taken in $H_n$. It follows from $\pm X\le Z\le Y$ that $Z=\abs{X}$. 

  To verify Krickeberg's formula, it suffices to show that $z_n=\sup_mE_n\abs{x_m}$. It follows from \cite[Lemma~2.3(4)]{Aliprantis:07} that $z_n=\sup_mE_n\abs{x_m}$ in $H_n$. Let $a\in F$ such that $E_n\abs{x_m}\le a$ for all $m\ge n$. Let $G$ be the subspace of $F$ spanned by $H_n$ and $a$. Again, we may view it as an ordered Banach space with closed positive cone $G_+=F_+\cap G$; $H_n$ is a closed subspace of $G$. Hence, we still have
\begin{math}
  \lim_mE_n\abs{x_m}=z_n
\end{math}
in $G$. Applying \cite[Lemma~2.3(4)]{Aliprantis:07} to $G$, we conclude that  $z_n=\sup_mE_n\abs{x_m}$ in $G$, and, therefore, $z_n\le a$. 

  \eqref{lat-hom} It is easy to see that the sequence $\bigl(\abs{x_n}\bigr)$ is a martingale and is the modulus of $X$. Also, for every fixed $n$ and every $m\ge n$ we have $E_n\abs{x_m}=\bigabs{E_nx_m}=\abs{x_n}$, so that  Krickeberg's formula s valid.

\end{proof}

It is an open problem whether the modulus of an operator is always given by the Riesz-Kantorovich formula, see~\cite[p.~59]{Aliprantis:07} for details. Similarly, it has been a natural conjecture that Krieckeberg's formula is always valid whenever $\mathcal M_r$ is a vector lattice. However, we will present a counterexample to the contrary. Our example will be based on \cite[Example~6]{Gessesse:11}, which we outline here for convenience of the reader.

\begin{example}(\cite{Gessesse:11})\label{halves}
Let $F=\mathbb{R}^\mathbb{N}$. For $n=0,1,2,\dots$, define $E_n$ via
\begin{multline*}
  E_n\bigl((a_i)\bigr)=
  \bigl(a_1,a_2,\dots,a_{2n},
  \frac{a_{2n+1}+a_{2n+2}}{2},\frac{a_{2n+1}+a_{2n+2}}{2},\\
  \frac{a_{2n+3}+a_{2n+4}}{2},\frac{a_{2n+3}+a_{2n+4}}{2},\dots\bigr)
\end{multline*}
Let $X=(x_n)_{n=0}^\infty$ where
\begin{displaymath}
  x_n=\bigl(\underbrace{-1,1,\dots,-1,1}_{2n},0,0,\dots\bigr)
\end{displaymath}
(we take $x_0=0$).
It is easy to see that $(E_n)$ is a filtration on $F$ and $X$ is a martingale with respect to $(E_n)$.
\end{example}

\begin{example}\emph{$\mathcal{M}_r$ is a vector lattice, yet Krickeberg's formula fails.}
  Let $F=\ell_\infty$; let $(E_n)$ and $X=(x_n)$ be as in Example~\ref{halves}, $n=0,1,\dots$.  Let
 $\varphi\colon F\to\mathbb R$ be a Banach limit. Put $P=\varphi\otimes\one$. It is easy to see that $P$ is a rank-one projection and that $E_nP=P$ for every $n$. It follows that the sequence $\bigl(PE_0,E_0,E_1,E_2,\dots\bigr)$ is a filtration on $F$ and the sequence $X=(x_0,x_0,x_1,x_2,\dots)$ is a martingale with respect to this filtration. Note that $\mathcal M_r$ is a vector lattice by Theorem~\ref{Mr-VL}.

We claim that $\abs{X}=(\one,\one,\one,\one,\dots)$. Indeed, it is easy to see that the sequence $(\one,\one,\one,\one,\dots)$ is a martingale which dominates $\pm X$. Now suppose $\pm X\le Y$ for some martingale $Y$. Put $Y=(z,y_0,y_1,\dots)$. For every $n\ge 0$ and $m\ge n$ we have $y_m\ge\abs{x_m}$, so that $y_n=E_ny_m\ge E_n\abs{x_m}$. Note that
 \begin{displaymath}
   E_n\abs{x_m}=\abs{x_m}=
   \bigl(\underbrace{1,1,\dots,1,1}_{2m},0,0,\dots\bigr).
 \end{displaymath}
This yields $y_n\ge\one$ for every $n\ge 0$. It follows from $y_0\ge\one$ that $z=PE_0y_0\ge\one$. Hence, $Y\ge(\one,\one,\one,\one,\dots)$. Therefore,
$\abs{X}=(\one,\one,\one,\one,\dots)$.

However, Krickeberg's formula for the initial term gives $\sup_{m}PE_0\abs{x_m}=0$ instead of $\one$.
\end{example}

\section{The space of regular bounded martingales on a~Banach lattice}

We say that $(E_n)$ is \term{uniformly bounded} if $\sup_n\norm{E_n}<+\infty$; we say that  $(E_n)$ is \term{contractive} if $\norm{E_n}\le 1$ for every $n$. A martingale $X=(x_n)$ in $\mathcal{M}\bigl(F,(E_n)\bigr)$ is said to be \term{bounded} if its \term{martingale norm} defined by $\norm{X}=\sup_n\norm{x_n}$ is finite. We denote by $M=M\bigl(F,(E_n)\bigr)$ the space of all bounded martingales on $F$ with respect to the filtration $(E_n)$. It is easy to see that $M$ is a closed subspace of $\ell_\infty(F)$; hence $M$ is a Banach space. It can be easily verified that the martingale norm is \term{monotone}, i.e., $0\le X\le Y$ implies $\norm{X}\le\norm{Y}$. The space of regular bounded martingales is the following subspace of $M$:
\begin{displaymath}
  M_r=M_r\bigl(F,(E_n)\bigr)=
  \bigl\{X_1-X_2 \mid X_1,X_2 \in M_+\bigr\}.
\end{displaymath}
Again, one may expect similarities with the well-known theory of regular operators; see, e.g., \cite{Wickstead:07,Xanthos:15}. It is well known that every regular operator on a Banach lattice is bounded and the space of regular operators on an order complete Banach lattice is a Banach lattice under the \emph{regular norm}. We will prove that if $F$ is order complete then $M_r$ is a Banach lattice under the regular norm. We will show that, in contrast to the setting of regular operators, in general $M_r \neq  \mathcal{M}_r$. Furthermore, $F$ is a KB-space iff $F$ is order continuous and every bounded martingale with respect to every uniformly bounded filtration is regular. 

\begin{example}
  \emph{Positive unbounded martingale on $\ell_1$.}
  For any $0\le \alpha\le 1$, define
  \begin{math}
    P_{\alpha}=
    \Bigl[
    \begin{smallmatrix}
      0 & 0 \\ \alpha & 1
    \end{smallmatrix}
    \Bigr].
  \end{math}
  It is easy to see that $P_{\alpha}$ is a positive projection onto $e_2$, and $P_\alpha$ is a contraction when viewed as an operator on $\ell_1^2$.  Define a filtration on $\ell_1$ as follows.
  \begin{multline*}
  E_1=
  \begin{bmatrix}
      0 & 0 &  &  &  &  &    \\
      1 & 1 &  &  &  && \\
      & & 0 & 0 &&& \\
      & & \frac12 & 1 &&& \\
      & & & & 0 & 0 & \\
      & & & & \frac14 & 1 & \\
      & & & & & & \ddots \\
  \end{bmatrix},
  \quad
  E_2=
  \begin{bmatrix}
      1 & 0 &  &  &  &  &  &  \\
      0 & 1 &  &  &  &&& \\
      & & 0 & 0 &&& \\
      & & \frac12 & 1 &&& \\
      & & & & 0 & 0 & \\
      & & & & \frac14 & 1 & \\
      & & & & & &  \ddots \\
  \end{bmatrix},\\
  E_3=
  \begin{bmatrix}
      1 & 0 &  &  &  &  &  &  \\
      0 & 1 &  &  &  &&& \\
      & & 1 & 0 &&& \\
      & & 0 & 1 &&& \\
      & & & & 0 & 0 & \\
      & & & & \frac14 & 1 & \\
      & & & & & & \ddots \\
  \end{bmatrix},
  \text{ etc.}
  \end{multline*}
  It is easy to see that this is a filtration $\norm{E_n}=1$. Further, define
  \begin{eqnarray*}
    x_1&=&(0,1,\ 0,\tfrac12,\ 0,\tfrac14,\ 0,\tfrac18,\dots),\\
    x_2&=&(1,0,\ 0,\tfrac12,\ 0,\tfrac14,\ 0,\tfrac18,\dots),\\
    x_3&=&(1,0,\ 1,0,\ 0,\tfrac14,\ 0,\tfrac18,\dots),\\
    x_4&=&(1,0,\ 1,0,\ 1,0,\ 0,\tfrac18,\dots),\\
    \text{etc.}&&
  \end{eqnarray*}
  It can be easily verified that $(x_n)$ is a positive martingale with respect to the filtration $(E_n)$, but $\norm{x_n}>n-1$, so this martingale is unbounded.
\end{example}

On $M_r$, we define the following so called \term{regular norm}:
$$\norm{X}_r=\inf\bigl\{\norm{Y}\mid Y \in M_+,\ Y \ge\pm X\bigr\}.$$

We claim that the space $\bigl(M_r,\norm{\cdot}_r\bigr)$ is a Banach space. We will prove a more general result for ordered Banach spaces. In particular this result is known to be true for the space of regular operators and the space generated by positive compact operators (\cite{Chen:97}, Proposition 2.2).

\begin{theorem}\label{reg-norm}
  Suppose that $\bigl(X,\norm{\cdot}\bigr)$ is an ordered normed space with a closed cone $X_+$ and a monotone norm. Then the following formula defines a norm on $X_r=X_+-X_+$:
$$\norm{x}_r=\inf\bigl\{\norm{y}\mid y \in X_+,\ y \ge\pm x\bigr\}.$$
For every $z \in X_r$, we have $\norm{z}\le 2\norm{z}_r$. Moreover, if $X$ is complete then $\bigl(X,\norm{\cdot}_r\bigr)$ is complete and if, in addition, $X=X_r$ then $\norm{\cdot}$ and $\norm{\cdot}_r$ are equivalent. If $X_r$ is a vector lattice then $\norm{x}_r=\bignorm{\abs{x}}$ for all $x\in X_r$.
\end{theorem}

\begin{proof}
  It is easy to see that $\norm{\cdot}_r$ is positively
  homogeneous. To verify the triangle inequality, let $u,v\in X_r$,
  take any $\varepsilon>0$, and find $x,y\in X_+$ such that $-x\le
  u\le x$ and $-y\le v\le y$, $\norm{x}\le\norm{u}_r+\varepsilon$, and
  $\norm{y}\le\norm{v}_r+\varepsilon$. It follows that
  \begin{math}
    -(x+y)\le u+v\le x+y,
  \end{math}
  so that
  \begin{displaymath}
    \norm{u+v}_r\le\norm{x+y}\le\norm{x}+\norm{y}\le\norm{u}_r+\norm{v}_r+2\varepsilon.
  \end{displaymath}
  This yields $\norm{u+v}_r\le\norm{u}_r+\norm{v}_r$.

  Fix $z\in X_r$. Let $x\in X_+$ be such that $-x\le z\le x$. Then $0\le x\pm z\le 2x$. It follows that
  \begin{displaymath}
    \norm{z}=\bignorm{\tfrac12(x+z)-\tfrac12(x-z)}
    \le\tfrac12\norm{x+z}+\tfrac12\norm{x-z}
    \le\tfrac12\norm{2x}+\tfrac12\norm{2x}
    =2\norm{x}.
  \end{displaymath}
  Taking the infimum
  over all such $x$, we get $\norm{z}\le2\norm{z}_r$. In particular,
  if $\norm{z}_r=0$ then $\norm{z}=0$ and, therefore, $z=0$.

  Now suppose that $\bigl(X,\norm{\cdot}\bigr)$ is complete and show
  that $\bigl(X_r,\norm{\cdot}_r\bigr)$ is complete. Let $(z_n)$ be a
  $\norm{\cdot}_r$-Cauchy sequence in $X_r$. Note that since $\norm{\cdot}\le
  2\norm{\cdot}_r$, the sequence is also Cauchy in the original norm,
  hence $z_n\xrightarrow{\norm{\cdot}}z$ for some $z\in X$. It suffices
  to show that $z\in X_r$ and some subsequence of $(z_n)$ converges to
  $z$ in $\norm{\cdot}_r$ because, in this case, the entire sequence
  would still converge to $z$ in $\norm{\cdot}_r$.

  Without loss of generality, passing to a subsequence, we may assume
  that for every $n$ and every $k\ge n$ we have
  $\norm{z_n-z_k}_r\le\frac{1}{3^n}$. For each $n$, $z_{n+1}-z_n\in
  X_r$, so we can find $x_n\in X_+$ such that $-x_n\le z_{n+1}-z_n\le
  x_n$ and $\norm{x_n}\le\frac{1}{2^n}$. Fix $m$. It follows from
  $z_n\xrightarrow{\norm{\cdot}}z$ that
  $z-z_m=\sum_{n=m}^\infty(z_{n+1}-z_n)$, where the series converges
  in $\norm{\cdot}$. Note that
  \begin{displaymath}
    \sum_{n=m}^k(z_{n+1}-z_n)\le\sum_{n=m}^kx_n
  \end{displaymath}
  for every $k>m$ and $\sum_{n=m}^\infty x_n$ converges in $\norm{\cdot}$. Since $X_+$
  is closed, $z-z_m\le\sum_{n=m}^\infty x_n$.
  Similarly, $-(z-z_m)\le\sum_{n=m}^\infty x_n$.
  It follows that
  \begin{displaymath}
    \norm{z-z_m}_r\le\Bignorm{\sum_{n=m}^\infty x_n}
    \le\sum_{n=m}^\infty\norm{x_n}
    \le\tfrac{1}{2^{m-1}}\to 0.
  \end{displaymath}
  If, moreover, $X=X_r$, we have that $\norm{\cdot}$ and $\norm{\cdot}_r$ are two complete norms on the same space with one of them dominating the other; it follows that the norms are equivalent.

  Suppose that $X_r$ is a vector lattice and $x\in X_r$. It follows from $\abs{x}\ge\pm x$ that $\bignorm{\abs{x}}\ge\norm{x}_r$. On the other hand, if $\pm x\le y$ for some $y\in X_+$ then $\abs{x}\le y$ and the monotonicity of norm yields $\bignorm{\abs{x}}\le\norm{y}$, hence $\bignorm{\abs{x}}\le\norm{x}_r$.
\end{proof}

\begin{corollary}\label{Mr-BS}
  Let $F$ be a Banach lattice and $(E_n)$ a filtration on $F$. The space $\bigl(M_r,\norm{\cdot}_r\bigr)$ is a Banach space and $\norm{X}\le\norm{X}_r$ for every $X\in M_r$. If $M_r$ is a vector lattice then $\norm{X}_r=\bignorm{\abs{X}}$ for all $X\in M_r$.
\end{corollary}

\begin{proof}
Applying Theorem \ref{reg-norm} to $M$, we conclude that $\bigl(M_r,\norm{\cdot}_r\bigr)$ is a Banach space and that if $M_r$ is a vector lattice then $\norm{X}_r=\bignorm{\abs{X}}$ for all $X\in M_r$. For $Y \in M_+$ such that $\pm X\le Y$, by the definition of the martingale norm we have $\norm{X}\le\norm{Y}$. It follows that $\norm{X}\le\norm{X}_r$.
\end{proof}

The regular norm may coincide with the martingale norm on $M_r$. We recall here that a Banach lattice $F$ is said to have the \term{Fatou property} if $0\le x_\alpha\uparrow x$ implies $\norm{x_\alpha}\to\norm{x}$. Dual and order continuous Banach lattices enjoy the Fatou property; see, e.g., \cite[p.~96]{Meyer-Nieberg:91} or \cite[p.~65]{Abramovich:02}.

\begin{proposition}\label{Fatou}
Let $F$ be a Banach lattice with the Fatou property and $(E_n)$ a contractive filtration on $F$. If $M_r$ is a vector lattice and Krickeberg's formula is valid on $M_r$ then  $\norm{X}=\norm{X}_r$ for every $X \in M_r$.
\end{proposition}

\begin{proof}
Let $X=(x_n) \in M_r$ and $Z=(z_n)=\abs{X}$. By Corollary~\ref{Mr-BS},
$\norm{X}_r=\norm{Z}\ge\norm{X}$. We need to show that $\norm{X}\ge\norm{Z}$. By Krickeberg's formula, for each $n \in \mathbb{N}$ we have $E_n\abs{x_m}\uparrow z_n$ (in~$m$).  The Fatou property yields $\norm{z_n}=\sup_m\bignorm{E_n\abs{x_m}}\le\sup_m\norm{x_m}=\norm{X}$, hence $\norm{Z}\le\norm{X}$.
\end{proof}

The following is the main result of our paper. Note that in view of Proposition \ref{Fatou} this result extends Theorem 13 in \cite{Troitsky:05}. Recall that if $F$ is order complete then $\mathcal M_r$ is a vector lattice by Theorem~\ref{Mr-VL}.

\begin{theorem}\label{Mr-BL}
  Let $F$ be an order complete Banach lattice and $(E_n)$ a filtration on $F$. Then the space $M_r$ is an ideal of $\mathcal{M}_r$ and a Banach lattice under the regular norm.
\end{theorem}

\begin{proof}
Note first that if $X\in M_r$ and $Y\in \mathcal{M}_r$ such that $0\le Y\le X$ then $Y\in M_r$. Let $X\in M_r$. Then there exists $Y\in M_r$ such that $Y\ge\pm X$. By Theorem~\ref{Mr-VL}, $\abs{X}$ exists in $\mathcal{M}_r$. Clearly, $\abs{X}\le Y$ and, therefore, $\abs{X}\in M_r$. Hence, $M_r$ is a vector lattice and an ideal of $\mathcal{M}_r$.

By Corollary~\ref{Mr-BS}, $\bigl(M_r,\norm{\cdot}_r\bigr)$ is a Banach lattice.
\end{proof}

\begin{example} \emph{$M_r$ need not be a Banach lattice under the martingale norm.}
 Let $F=\ell_\infty$, equipped with following equivalent norm:
 $$\bignorm{(a_i)} =\bignorm{(a_i)}_\infty+\limsup\abs{a_i}$$

It can be easily verified that $\bigl(F,\norm{\cdot}\bigr)$ is a Banach lattice. Let $(E_n)$ and $X=(x_n)$ be as in Example~\ref{halves}. Note that each $E_n$ is order continuous. Clearly, $M=M_r$. By Theorem~\ref{Mr-BL}, $M_r$ is a Banach lattice under $\norm{\cdot}_r$ and an ideal in $\mathcal M_r$. For each $n$, we have
  \begin{displaymath}
    E_n\abs{x_m}=\abs{x_m}=(\underbrace{1,1,\dots,1,1}_{2m},0,\dots)\uparrow\one,
  \end{displaymath}
  By Proposition~\ref{Krickeberg}\eqref{oc-oc}, lattice operations are given by Krickeberg's formula. It follows that the modulus of $X$ is the constant martingale $\abs{X}_n=\one$. We have $\bignorm{\abs{X}}=2$ and $\norm{x_n}=1$ for each $n$, thus $1=\norm{X}<\bignorm{\abs{X}}$, so that  $M_r$ fails to be a Banach lattice under the martingale norm.
\end{example}

Finally, we study under which conditions we have $M=M_r$. Recall that a vector lattice $F$ has a \term{strong unit} $e$ whenever $F=\bigcup_{n=1}^\infty [-ne,ne]$. A Banach lattice $F$ has an \term{order continuous norm} whenever $x_\alpha\downarrow 0$ implies $x_\alpha\to 0$; $F$ is a \term{KB-space} if it does not contain a sublattice isomorphic to $c_0$.

\begin{proposition}
  Let $F$ be a Banach lattice with a strong unit $e$ and $(E_n)$ a uniformly bounded filtration on $F$. Then $M=M_r$. If, in addition, $F$ is order complete then $M$ is a vector lattice with a strong unit.
\end{proposition}

\begin{proof}
  It is known that the original norm of $F$ is equivalent to the norm $\norm{\cdot}_\infty$ generated by $e$; see, e.g., \cite[p.~194]{Aliprantis:06}. In particular, there exists $C>0$ such that $\abs{x}\le C\norm{x}e$ for every $x\in F$.

  For each $n$, put $y_n=E_ne$. Clearly, $Y=(y_n)$ is a bounded positive martingale. Let $X\in M$. Then for every $n$ we have $\pm x_n\le C\norm{x_n}e\le C\norm{X}e$. Applying $E_n$, we get $\pm x_n\le C\norm{X}y_n$. It follows that $\pm X\le C\norm{X}Y$, so that $X$ is regular. Hence, $M=M_r$. If, in addition, $F$ is order complete then $M$ is a vector lattice by Theorem~\ref{Mr-BL}. It follows from $\pm X\le C\norm{X}Y$ that $Y$ is a strong unit in $M$.
\end{proof}

\begin{theorem}
Let $F$ be a Banach lattice with an order continuous norm. Then the following are equivalent.
\begin{enumerate}
\item\label{KB-KB} $F$ is a KB-space;
\item\label{KB-MMr} $M=M_r$ for every uniformly bounded filtration $(E_n)$ on $F$;
\item\label{KB-M-VL} $M$ is a vector lattice for any uniformly  bounded filtration $(E_n)$ on $F$.
\end{enumerate}
\end{theorem}

\begin{proof}
  \eqref{KB-M-VL}$\Rightarrow$\eqref{KB-MMr} trivially. \eqref{KB-MMr}$\Rightarrow$\eqref{KB-M-VL} by Theorem~\ref{Mr-BL}

\eqref{KB-KB}$\Rightarrow$\eqref{KB-M-VL} The proof is similar to that of \cite[Theorem~7]{Troitsky:05}.  Let $X=(x_n) \in M$. Fix some $n \in \mathbb{N}$. The sequence $\bigl(E_n\abs{x_m}\bigr)_{m=n}^\infty$ is increasing and norm bounded. Since $F$ is a KB-space, it follows that $z_n=\lim_m E_n\abs{x_m}$ exists; then $Z$ is a martingale and $Z=\abs{X}$.

\eqref{KB-MMr}$\Rightarrow$\eqref{KB-KB} Suppose that $F$ is not a KB-space. Then $c_0$ is lattice embeddable in $F$. Without lose of generality, we can assume that $c_0$ is a closed sublattice of $F$. Let $(E_n)$ and $X=(x_n)$ be as in Example~\ref{halves}; view $(E_n)$ as a (uniformly bounded) filtration on $c_0$ and $X$ as a martingale in $c_0$.

By \cite[Corollary~2.4.3]{Meyer-Nieberg:91}, there exist a positive projection $P:F \rightarrow c_0$. It is easy to see that $(PE_n)$ is again a uniformly bounded filtration on $F$ and $X$ is a bounded martingale with respect to it. We claim that $X$ is not regular. Indeed, suppose that $\pm X\le Z$ for some positive martingale $Z=(z_n)$ in $F$. Then $\pm x_n\le z_n$ for every $n$ yields $\abs{x_n}\le z_n$, so that
\begin{math}
  E_0\abs{x_n}=PE_0\abs{x_n}\le PE_0z_n=z_0.
\end{math}
It follows that the increasing sequence $\bigl(E_0\abs{x_n}\bigr)$ is bounded above in $F$, hence it converges in $F$ because $F$ is order continuous. Therefore, $(E_0\abs{x_n})$ converges in $c_0$, which is clearly false.
\end{proof}

\medskip

\textbf{Acknowledgements.} We would like to thanks O.Blasco and N.Gao for helpful discussions.

\end{document}